\documentclass[reqno]{amsart}
\usepackage{epsf}
\usepackage{todonotes}
\usepackage{verbatim}
\usepackage{amsmath}
\usepackage{amsopn}
\usepackage{amssymb}
\usepackage{latexsym}
\usepackage{amsfonts}
\usepackage{amsthm}
\usepackage{tikz}

\newcommand{\be}{{\beta}}
\newcommand{\g}{{\gamma}}
\newcommand{\Des}{{\rm{Des}}}
\newcommand{\des}{{\rm{des}}}

\newtheorem{thm}{Theorem}[section]
\newtheorem{pro}[thm]{Proposition}
\newtheorem{lem}[thm]{Lemma}
\newtheorem{cla}[thm]{Claim}

\newtheorem{cor}[thm]{Corollary}

\theoremstyle{definition}
\newtheorem{obs}[thm]{Observation}

\newtheorem{rem}[thm]{Remark}
\newtheorem{exa}[thm]{Example}

\newtheorem{defn}[thm]{Definition}
\newtheorem{definition}[thm]{Definition}

\newtheorem{prop}[thm]{Proposition}
\newtheorem{conj}[thm]{Conjecture}

\newcommand{\een}{\end{enumerate}}
\newcommand{\blem}{\begin{lem}}
\newcommand{\elem}{\end{lem}}
\newcommand{\bcl}{\begin{cla}}
\newcommand{\ecl}{\end{cla}}
\newcommand{\ethm}{\end{thm}}
\newcommand{\bpr}{\begin{pro}}
\newcommand{\epr}{\end{pro}}
\newcommand{\bco}{\begin{cor}}
\newcommand{\eco}{\end{cor}}
\newcommand{\bcon}{\begin{conj}}
\newcommand{\econ}{\end{conj}}
\newcommand{\bde}{\begin{defn}}
\newcommand{\ede}{\end{defn}}
\newcommand{\bex}{\begin{exa}}
\newcommand{\eexa}{\end{exa}}
\newcommand{\bobs}{\begin{obs}}
\newcommand{\eobs}{\end{obs}}
\newcommand{\bexe}{\begin{exe}}
\newcommand{\eexe}{\end{exe}}

\newcommand{\R}{{\mathbb{R}}}

\newcommand{\grn}{G_{m,n}}

\begin{document}
\title[Identities involving Stirling numbers of types $B$ and $D$]{Some identities involving second kind Stirling numbers of types $B$ and $D$}

\thanks{This research was supported by a grant from the Ministry of Science and Technology, Israel, and the France's Centre National pour la Recherche Scientifique (CNRS)}

\author{Eli Bagno, Riccardo Biagioli and David Garber}

\address{Eli Bagno, Jerusalem College of Technology\\
21 Havaad Haleumi St. Jerusalem, Israel}
\email{bagnoe@g.jct.ac.il}

\address{Riccardo Biagioli\\
Institut Camille Jordan, Universit\'e Claude Bernard Lyon 1 \\
69622 Villeurbanne Cedex, France}
\email{biagioli@math.univ-lyon1.fr}

\address{David Garber \\
Department of  Applied Mathematics, Holon Institute of Technology,
52 Golomb St., PO Box 305, 58102 Holon, Israel,  and (sabbatical:) Einstein Institute of Mathematics, Hebrew University of Jerusalem, Jerusalem, Israel }
\email{garber@hit.ac.il}

\date{\today}

\maketitle

\begin{abstract}

Using Reiner's definition of Stirling numbers of the second kind in types $B$ and $D$, we generalize two well-known identities concerning the classical Stirling numbers of the second kind. The first identity relates them with Eulerian numbers and the second identity interprets them as entries in a transition matrix between the elements of two standard bases of the polynomial ring $\mathbb{R}[x]$. Finally, we generalize these identities to the group of colored permutations $G_{m,n}$.

\end{abstract}


\section{Introduction}
The {\em Stirling number of the second kind}, denoted $S(n,k)$, is defined as the number of partitions of the set $[n]:=\{1,\dots,n\}$ into $k$ non-empty subsets (see \cite[page 81]{EC1}). Stirling numbers of the second kind arise in a variety of problems  in enumerative combinatorics; they have many combinatorial interpretations, and have been generalized in various contexts and in different ways.

In the geometric theory of Coxeter groups they appear as follows. For any finite Coxeter group $W$, there is a corresponding hyperplane arrangement $\mathcal{W}$, whose elements are the reflecting hyperplanes of $W$. Associated with $\mathcal{W}$, there is the set of all the intersections of these hyperplanes, ordered by reverse inclusion, called the {\em intersection lattice}, and denoted $L(\mathcal{W})$ (see e.g. \cite{BS,Stanley-arr}). It is well-known that in the Coxeter group of type $A$, $L(\mathcal{A}_n)$ is isomorphic to the lattice of the set partitions of $[n]$. By this identification, the subspaces of dimension $n-k$ are counted by $S(n,k)$. In this geometric context, Stirling numbers of the second kind are usually called Whitney numbers (see \cite{Su,Za} for more details).

For Coxeter groups of types $B$ and $D$, Zaslavsky \cite{Za} gave a description of $L(\mathcal{B}_n)$ and $L(\mathcal{D}_n)$ by using the general theory of signed graphs. Then, Reiner~\cite{R} gave a different combinatorial representation of $L(\mathcal{B}_n)$ and $L(\mathcal{D}_n)$ in terms of new types of set partitions, called $B_n$- and {\em $D_n$-partitions}. We call the number of $B_n$- (resp. $D_n$-) partitions with $k$ pairs of nonzero blocks the  {\em Stirling number of the second kind of type} $B$ (resp. {\em type} $D$).
\smallskip

The posets of $B_n$- and $D_n$-partitions, as well as their isomorphic intersection lattices, have been studied in several papers~\cite{BjS, BjW, BS, CW1, CW2, Su},  from algebraic, topological and combinatorial points of view. However, to our knowledge, two famous identities concerning the classical Stirling numbers of the second kind (see e.g. Bona~\cite[Theorems 1.8 and 1.17]{Bo}) have not been generalized to types $B$ and $D$ in a combinatorial way: the first identity relates the Stirling numbers to the Eulerian numbers, and the second one  formulates a change of bases in $\R[x]$, both will be described below.

\smallskip

The original definition of the {\em Eulerian numbers} was given by Euler in an analytic context \cite[\S 13]{Eu}. Later, they began to appear in combinatorial problems, as the Eulerian number $A(n,k)$ counts the number of permutations in the symmetric group $S_n$ having $k-1$ descents, where a {\it descent} of $\sigma \in S_n$ is an element of the {\em descent set} of $\sigma$, defined by :
\begin{equation}\label{des_typeA}
{\rm Des}(\sigma):=\{i \in [n-1]\mid \sigma(i)>\sigma(i+1)\}.
\end{equation}
We denote by ${\des}(\sigma):=|\Des(\sigma)|$ the {\em descent number}.

\smallskip

The first above-mentioned identity relating Stirling numbers of the second kind and Eulerian numbers is the following one, see e.g. \cite[Theorem 1.17]{Bo}:

\begin{thm}\label{thm:typeA}
For all non-negative integers $n \geq r$, we have
\begin{equation}\label{eq:Stirling_Eulerian_typeA}
S(n,r) = \frac{1}{r!}\sum_{k=0}^r A(n, k) \binom{n-k}{r-k}.
\end{equation}
\end{thm}

The second identity arises when one expresses the standard basis of the polynomial ring $\R[x]$ as a linear combination of the basis consisting of the falling factorials (see e.g. the survey of Boyadzhiev \cite{Boy}):

\begin{thm}\label{thm:typeA_falling}
Let $x \in \mathbb{R}$ and let $n \in \mathbb{N}$. Then we have
\begin{equation}\label{reg}
x^n=\sum\limits_{k=0}^n{S(n,k)[x]_k},
\end{equation}
\noindent
where $[x]_k:=x(x-1) \cdots (x-k+1)$ is the {\em falling factorial of degree $k$} and $[x]_0:=1$.
\end{thm}

There are some known proofs for the last identity. A combinatorial one, realizing $x^n$ as the number of functions from the set  $\{1,\dots,n \}$ to the set $\{1,\dots,x \}$ (for a positive integer $x$), is presented in \cite[Eqn. (1.94d)]{EC1}.
The first geometric proof is due to Knop \cite{K}.

\smallskip

In this paper, we use Stirling numbers of the second kind of types $B$ and $D$, in order to generalize the identities stated in  Equations \eqref{eq:Stirling_Eulerian_typeA} and \eqref{reg}. Theorems~\ref{main_thm_B} and \ref{main_thm_D} below are generalizations of the first identity for types $B$ and $D$: they will be proven by providing explicit procedures to construct ordered set partitions starting from the elements of the corresponding Coxeter groups.

Theorems~\ref{thm_bala} and~\ref{thm_bala_D} generalize the second identity. We present here a geometric approach, suggested to us by Reiner~\cite{R1}, which is based on some geometric characterizations of the intersection lattices of types $B$ and $D$. 
Moreover, we show how to generalize these two classical identities to the colored permutations group $G_{m,n}$.

\smallskip

The rest of the paper is organized as follows.  Sections \ref{Eulerian numbers} and \ref{Set partitions} present the known generalizations of the Eulerian numbers and  the set partitions, respectively, to the Coxeter groups of types $B$ and $D$. In Sections \ref{connections between Stirling and Euler} and  \ref{Falling polynomials for Coxeter groups}, we state our generalizations of the two identities and  prove them.
Finally, in Section \ref{Possible generalizations}, we present some possible extensions of the main results.

\section{Eulerian numbers of types $B$ and $D$} \label{Eulerian numbers}

We start with some notations. For $n\in \mathbb{N}$, we set $[\pm n]:=\{\pm 1,\ldots,\pm n\}$. For a subset $B \subseteq [\pm n]$, we denote by $-B$ the set obtained by negating all the elements of $B$, and by $\pm B$ we denote the unordered pair of sets $B,-B$.

\medskip

Let $(W,S)$ be a Coxeter system. As usual, denote by $\ell(w)$ the {\em length} of $w \in W$, which is the minimal integer $k$ satisfying $w=s_1\cdots s_k$ with $s_i \in S$.
The {\em (right) descent set} of $w \in W$ is defined to be
$${\rm Des}(w):=\{s \in S \mid \ell(ws)<\ell(w)\}.$$
A combinatorial characterization of ${\rm Des}(w)$ in type $A$ is given by Equation \eqref{des_typeA} above. Now we recall analogous descriptions in types $B$ and $D$.

\smallskip

We denote by $B_{n}$ the group of all bijections $\be$ of the set
$[\pm n]$ onto itself such that
\[\be(-i)=-\be(i)\]
for all $i \in [\pm n]$, with composition as the
group operation. This group is usually known as the group of {\em
signed permutations} on $[n]$.
If $\be \in B_{n}$, then we write $\be=[\be(1),\dots,\be(n)]$ and we call this the
{\em window} notation of $\be$.

As a set of generators for $B_n$ we
take $S_B:=\left\{ s_0^B, s_1^B,\ldots,s_{n-1}^B \right\}$ where for $i \in[n-1]$
\[s_i^B:=[1,\ldots,i-1,i+1,i,i+2,\ldots,n] \;\; {\rm and} \;\; s_0^B:=[-1,2,\ldots,n].\]
It is well-known that $(B_n,S_B)$ is a Coxeter system of type $B$ (see e.g. \cite[\S 8.1]{BB}). The following characterization of the (right) descent set of $\be \in B_n$ is well-known \cite{BB}.

\begin{prop} Let $\be \in B_n$. Then
\begin{eqnarray*}
\Des_B(\be)=\{i \in [0,n-1] \mid \be(i) > \be(i+1)\},
\end{eqnarray*}
where  $\be(0):=0$ (we use the usual order on the integers). In particular, $0 \in \Des_B(\be)$ is a descent if and only if $\beta(1) < 0$. We set $\des_B(\be):=|\Des_B(\be)|.$
\end{prop}

For all non-negative integers $n\geq k$, we set
\begin{equation}\label{def:Eulerian_B}
A_B(n,k):=|\{\be \in B_n \mid \des_B(\be)= k \}|,
\end{equation}
and we call them the {\em Eulerian numbers of type} $B$.

Note that in our context $A_B(n,k)$ counts permutations in $B_n$ having $k$ descents rather than $k-1$, like in type $A$, since this produces nicer formulas.

\medskip

We denote by $D_{n}$ the subgroup of $B_{n}$ consisting of all the
signed permutations having an {\em even} number of negative entries in
their window notation.
It is usually called the {\em even-signed permutation group}. As a set of generators for $D_n$ we take
$S_D:=\left\{ s_{0}^D,s_{1}^D,\dots,s_{n-1}^D \right\}$ where for $i \in [n-1]$:
\[s_i^D:=s_i^B \;\; {\rm and} \;\; s_{0}^D:=[-2,-1,3,\ldots,n].\]
It is well-known that $(D_n,S_D)$ is a Coxeter system of type $D$, and there is a direct combinatorial way to compute the (right) descent set of $\g \in D_{n}$ (see e.g. \cite[\S 8.2]{BB}):
\begin{prop}Let $\g \in D_n$. Then
\begin{eqnarray*}\label{lD}
\Des_D(\g)=\{i \in [0,n-1] \mid \g(i)>\g(i+1)\},
\end{eqnarray*}
where $\g(0):=-\g(2)$. In particular,  $0 \in \Des_D(\gamma)$ if and only if\break $\gamma(1)+\gamma(2)<0$.   We set $\des_D(\g):=|\Des_D(\g)|$.
\end{prop}

For all non-negative integers $n\geq k$, we set:
\begin{equation}\label{def:Eulerian_D}
A_D(n,k):=|\{\g \in D_n \mid \des_D(\g)=k\}|,
\end{equation}
and we call them the {\em Eulerian numbers of type} $D$.
\smallskip

For example, if $\g=[1,-3,4,-5,-2,-6]$, then:
$$\Des_D(\g)=\{0,1,3,5\} \mbox{, but } \Des_B(\g)=\{1,3,5\}.$$

\section{Set partitions of types $B$ and $D$}\label{Set partitions}

In this section, we introduce the concepts of set partitions of types $B$ and $D$ as defined by Reiner \cite{R}.

As mentioned above, we denote by $L(\mathcal{W})$
the intersection lattice corresponding to the Coxeter hyperplane arrangement $\mathcal{W}$ of a finite Coxeter group $W$.
We will focus only on the hyperplane arrangements of types $A$, $B$ and
$D$. In terms of the coordinate functions $x_1,\dots,x_n$ in $\R^n$, they can be defined as follows:
\begin{eqnarray*}
\mathcal{A}_n &:=& \{\ \{ x_i = x_j\} \mid 1 \leq i < j \leq n \},\\
\mathcal{B}_n &:=&  \{\ \{ x_i = \pm x_j\} \mid 1 \leq i < j \leq n \}  \cup \{\ \{ x_i = 0\} \mid 1 \leq  i \leq n\},\\
\mathcal{D}_n &:=& \{\ \{ x_i = \pm x_j\} \mid 1 \leq i < j \leq n \}.
\end{eqnarray*}

It is well-known that in type $A$, the intersection lattice $L(\mathcal{A}_n)$ is isomorphic to the lattice of set partitions of $[n]$.

In type $B$, let us consider the following element of $L(\mathcal{B}_9)$:
$$\{x_1=-x_3=x_6=x_8=x_9, x_2=x_4=0, x_5=-x_7\}.$$
It can be easily presented as the following set partition of $[\pm 9]$:
$$\{\{1, -3,6,8,-9\},\{-1,3,-6,-8,9\},\{2,-2,4,-4\},\{5,-7\},\{-5, 7\}\}.$$
This probably was Reiner's motivation to define the set partitions of type $B$, as follows:

\begin{defn}
A {\it $B_n$-partition} is a set partition of $[\pm n]$ into blocks such that the following conditions are satisfied:
\begin{itemize}
\item There exists at most one block satisfying $-C=C$, called the {\em zero-block}. It is a subset of $[\pm n]$ of the form $\{\pm i \mid i \in S\}$ for some $S \subseteq [n]$.
\item If $C$ appears as a block in the partition, then $-C$ also appears in that partition.
\end{itemize}
\end{defn}

A similar definition holds for set partitions of type $D$:
\begin{definition}
A {\em $D_n$-partition} is a $B_n$-partition such that the zero-block, if exists, contains at least two positive elements.
\end{definition}

We denote by $S_B(n,r)$ ({\em resp.} $S_D(n,r)$) the number of  $B_n$- ({\em resp.}\break $D_n$-) partitions having exactly $r$ pairs of nonzero blocks. These numbers are called {\em Stirling numbers (of the second kind) of type $B$} ({\em resp. type} $D$).  They correspond, respectively,  to the sequences oeis.org/A039755 and oeis.org/A039760 in the OEIS. Tables~\ref{Table1} and~\ref{Table2} record these numbers for small values of $n$ and $r$.

\medskip

We now define the concept of an ordered set partition:
\begin{defn}
A $B_n$-partition (or $D_n$-partition) is {\em ordered} if the set of blocks is totally ordered and the following conditions are satisfied:
\begin{itemize}
\item If the zero-block exists, then it appears as the first block.
\item For each block $C$ which is not the zero-block, the blocks $C$ and $-C$ are adjacent.
\end{itemize}
\end{defn}

\begin{exa} The following partitions
\begin{eqnarray*}P_1 & = &\{ \{ \pm 3\}, \pm\{ -2,1 \}, \pm\{ -4,5 \} \}, \\
P_2 & = &\{ \pm\{ 1 \},\pm\{ 2 \}, \pm\{ -4,3 \} \},\\
P_3 &= &\left[ \{ \pm 1, \pm 3\}, \{ -2 \}, \{ 2 \},\{ -4,5 \},\{ -5,4 \} \right],
\end{eqnarray*}
are respectively, a $B_5$-partition which is not a $D_5$-partition, a $D_4$-partition with no zero-block, and an ordered $D_5$-partition having a zero-block.
\end{exa}
\begin{table}
\begin{center}
\begin{tabular}{r||r|r|r|r|r|r|r}
  $n/r$ & 0 & 1 & 2 & 3 & 4 & 5 & 6  \\
  \hline\hline
   0    & 1 &   &   &   &   &   &     \\
   1    & 1 & 1 &   &   &   &   &     \\
   2    & 1 & 4 & 1  &   &   &   &     \\
   3    & 1 & 13 & 9  & 1  &   &   &     \\
   4    & 1 & 40 & 58  & 16  & 1  &   &     \\
   5    & 1 & 121 & 330  & 170  & 25  & 1  &     \\
   6    & 1 & 364 & 1771  & 1520  & 395  & 36  & 1    \\
\end{tabular}
\end{center}
\bigskip

\caption{Stirling numbers $S_B(n,r)$ of the second kind of type $B$.}\label{Table1}
\end{table}

\begin{table}
\begin{center}
\begin{tabular}{r||r|r|r|r|r|r|r}
  $n/r$ & 0 & 1 & 2 & 3 & 4 & 5 & 6  \\
  \hline\hline
   0    & 1 &   &   &   &   &   &     \\
   1    & 0 & 1 &   &   &   &   &     \\
   2    & 1 & 2 & 1  &   &   &   &     \\
   3    & 1 & 7 & 6  & 1  &   &   &     \\
   4    & 1 & 24 & 34  & 12  & 1  &     &  \\
   5    & 1 & 81 & 190  & 110  & 20  & 1  &     \\
   6    & 1 & 268 & 1051  & 920  & 275  & 30  & 1    \\
\end{tabular}
\end{center}
\bigskip

\caption{Stirling numbers $S_D(n,r)$ of the second kind of type $D$.}\label{Table2}
\end{table}

\section{Connections between Stirling and Eulerian numbers of types $B$ and $D$}\label{connections between Stirling and Euler}

In this section, we present two generalizations of  Theorem \ref{thm:typeA} for Coxeter groups of types $B$ and $D$.

\begin{thm}\label{main_thm_B} For all non-negative integers $n \geq r$, we have:
$$ S_{B}(n,r)= \frac{1}{2^{r}r!} \sum\limits_{k=0}^r {A_B(n,k){\binom{n-k}{r-k}}}.$$
\end{thm}

\begin{thm}\label{main_thm_D} For all non-negative integers $n \geq r$, with $n\neq 1$,  we have:
$$ S_D(n,r) = \frac{1}{2^{r}r!} \left[  \sum\limits_{k=0}^r {A_D(n,k){\binom{n-k}{r-k}}} + n \cdot 2^{n-1}(r-1)! \cdot S(n-1,r-1)\right],$$
where $S(n-1,r-1)$ is the usual Stirling number of the second kind.
\end{thm}

Now, by inverting these formulas, similarly to the known equality in type $A$, mentioned in \cite[Corollary 1.18]{Bo}:
\begin{equation}\label{Ank}
A(n,k)= \sum\limits_{r=1}^k {(-1)^{k-r} \cdot  r! \cdot S(n,r) \cdot \binom{n-r}{k-r}},
\end{equation}
we get the following expressions for the  Eulerian numbers of type $B$ ({\em resp.} type $D$) in terms of the Stirling numbers of type $B$ ({\em resp.} type $D$):

\begin{cor}\label{inverse_main_thm_star} For all  non-negative integers $n \geq k$, we have:
$$A_B(n,k)= \sum\limits_{r=0}^k {(-1)^{k-r} \cdot  2^{r}r! \cdot S_B(n,r) \cdot \binom{n-r}{k-r}}.$$
\end{cor}

\begin{cor}\label{inverse_main_thm_star_D} For all non-negative integers $n\geq k$, with $n\neq 1$, we have:
\begin{small}
\begin{eqnarray*}\label{eq_D} \nonumber
A_D(n,k)
= \left[ \sum\limits_{r=0}^k (-1)^{k-r} \cdot 2^{r}r! \cdot S_D(n,r) \cdot \binom{n-r}{k-r} \right] - n \cdot 2^{n-1} \cdot A(n-1,k-1).
\end{eqnarray*}
\end{small}
\end{cor}

\subsection{Proof for type $B$}

The proofs in this and in the next subsections use arguments similar to Bona's proof for the corresponding identity for type $A$, see  \cite[Theorem 1.17]{Bo}.

\begin{proof}[Proof of Theorem~\ref{main_thm_B}]\label{section4.1}

We have to prove the following equality:
$$ 2^{r}r!S_{B}(n,r)= \sum\limits_{k=0}^r {A_B(n,k){\binom{n-k}{r-k}}}.$$

The number $2^r r!S_B(n,r)$ in the left-hand side is the number of ordered $B_n$-partitions having $r$ pairs of nonzero blocks. Now, let us show that the right-hand side counts the same set of partitions in a different way.

Let $\be \in B_n$ be a signed permutation with $\des_B(\be)=k$, written in its window notation.  We start by adding a separator after each descent of $\beta$ and after $\be(n)$. If $0 \in \Des_B(\be)$, this means that a separator is added before $\be(1)$. If $r>k$, we add extra $r-k$ {\em artificial separators} in some of the $n-k$ empty spots, where by a {\em spot} we mean a gap between two consecutive entries of $\be$ or the gap before the first entry $\be(1)$.
This splits $\be$ into a set of $r$ blocks,
where the block $C_i$ is defined as the set of entries between the $i$th and the $(i+1)$th separators for $1 \leq i \leq r$.
Now, this set of blocks is transformed into the ordered $B_n$-partition with $r$ pairs of nonzero blocks:
$$[C_0,C_1,-C_1,\ldots,C_r,-C_r],$$
where the (optional) zero-block $C_0$ equals to $\{\pm \be(1),\ldots, \pm \be(j)\}$ if the first separator is after $\be(j)$, for $j \geq 1$, and it does not exist if the first separator is before $\be(1)$.

For example, if $\beta=[-2,3,5,1,-4] \in B_5$, then after adding the separators induced by descents, we get the sequence
$[ \ | \ -2,3,5\ | \ 1 \ | \ -4 \ | \ ]$, which is transformed into the ordered partition $[\{-2,3,5\},\{2,-3,-5\},\{1\},\{-1\},\{-4\},\{4\}]$.
On the other hand, if $\beta'=[2,3,5,-1,-4] \in B_5$, then after adding the separators induced by the descents, we have
$\beta'=[ \ 2,3,5 \ | \ -1 \ | \ -4 \ | \ ]$, which gives rise to the ordered partition $[\{ \pm 2, \pm 3,\pm 5 \},\{-1\},\{1\},\{-4\},\{4\}]$, with a zero-block, and two nonzero blocks.

There are exactly $\binom{n-k}{r-k}$ ordered $B_n$-partitions obtained from $\be$ in this way. From now on, we refer to this process of creating $B_n$-partitions starting from a single signed permutation $\be$, as the {\em $B$-procedure}.

It is easy to see that the $B$-procedure applied to different signed permutations produces disjoint sets of ordered $B_n$-partitions;  therefore,
one can create in this way $\sum_{k=0}^r {A_B(n,k){\binom{n-k}{r-k}}}$  distinct ordered $B_n$-partitions with $r$ pairs of nonzero blocks.

Let us show that any ordered $B_n$-partition $\lambda=[C_0,C_1,-C_1,\ldots,C_r,-C_r],$ can be obtained through the $B$-procedure.
If $\lambda$ contains a zero-block $C_0$, then put the positive elements of $C_0$ in increasing order at the beginning of a sequence $\mathbf{S}$, and add a separator after them. Then, order increasingly the elements in each of the blocks $C_1,\ldots,C_r$, and write them sequentially in $\mathbf{S}$ (after the first separator if exists), by adding a separator after the last entry coming from each block. Reading the formed sequence $\mathbf{S}$ from the left to the right, one obtains the window notation of a signed permutation $\be$. Note that the number of descents in $\be$ is smaller than or equal to $r$, since the elements in each block are ordered increasingly. Now, it is clear that $\lambda$ can be obtained by applying the $B$-procedure to $\be$, where the artificial separators are easily recovered.
\end{proof}

\begin{exa}
The signed permutation
$$\be=[ \ 1,4 \mid -5,-3,2 \ | \ ] \in B_5$$
has $2$ as a descent. It produces the following ordered $B_5$-partition with one pair of nonzero blocks
$$[ \{\pm 1,\pm 4\}, \{-5,-3,2\}, \{5,3,-2\}],$$
and exactly ${\binom{4}{1}}$ ordered $B_5$-partitions with two pairs of nonzero blocks, namely:
\begin{eqnarray*}
&[ \{1,4\},\{-1,-4\}, \{-5,-3,2\}, \{5,3,-2\}],\\
&[ \{\pm 1\},\{4\},\{-4\}, \{-5,-3,2\}, \{5,3,-2\}],\\
&[ \{ \pm 1, \pm 4\}, \{-5\}, \{5\},\{-3,2\}, \{3,-2\}], \\
&[ \{ \pm 1,\pm 4\}, \{-5,-3\},\{5,3\},\{2\}, \{-2\}],
\end{eqnarray*}
obtained by placing one artificial separator before entries  $1,2,4$ and $5$, respectively. The other ordered partitions coming from $\be$ with more blocks   are obtained similarly.

\smallskip

Conversely, let
$$\lambda=[\{ \pm 1, \pm 4\}, \{5\},\{-5\}, \{-3,2\}, \{3,-2\}]$$
be an ordered $B_5$-partition. The corresponding signed permutation with the added separators is
$\be=[ \ 1,4  \parallel 5 \ | -  3,2 \ |  \ ] \in B_5$.
Note that although $C_1=\{5\}$ is a separate block, there is no descent between $4$ and $5$, meaning that $\lambda$ is obtained by adding an artificial separator in the spot between these two entries, denoted $\|$.
\end{exa}

\subsection{Proof for type $D$}

The proof of Theorem \ref{main_thm_D} is a bit more tricky. The basic idea is the same as before: obtaining the whole set of ordered $D_n$-partitions with $r$ pairs of nonzero blocks from elements in $D_n$ with at most $r$ descents. We will use the $B$-procedure presented in the previous subsection, with the addition of an extra step, to take care of the special structure of the $D_n$-partitions.
First of all, we recall that we might have $\Des_D(\g) \neq \Des_B(\g)$ for $\g \in D_n$, see an example at the end of Section \ref{Eulerian numbers}.

\medskip

Let $\g \in D_n$ be such that $\des_D(\g)=k$. We start by adding the separators after the $D$-descents of $\g$ and the artificial ones in case that $k<r$. Using the $B$-procedure, we transform $\g$, equipped with the set of separators, into a $B_n$-partition. The result is also a $D_n$-partition, except in the case when there is a separator (either induced by a $D$-descent or by an artificial addition) between $\g(1)$ and $\g(2)$, but not before $\g(1)$. In fact, in this case, we obtain an ordered $B_n$-partition with a zero-block containing exactly one pair of elements, which is not a $D_n$-partition.

Hence, only in this case, we slightly modify the algorithm as follows. First we toggle the sign of $\g(1)$ and move the separator from after $\g(1)$ to before it. We call this action the {\em switch operation}.
Then, we transform this new sequence of entries and separators into a $D_n$-partition by applying the $B$-procedure. We refer to this process of associating a permutation $\g \in D_n$ with the obtained set of ordered $D_n$-partitions, as the {\em $D$-procedure}.

\medskip

Before proving that this procedure indeed creates ordered $D_n$-partitions, we give an example of an element $\g \in D_n$, for which the application of the switch operation is required.

\begin{exa}
Let $\g=[ \ -1 \ \| \ 3, 4 \mid -2 \mid -6,-5 \mid \ ] \in D_6$ be a permutation equipped with the separators induced by the $D$-descents and one artificial separator added after position $1$. The $B$-procedure, applied to $\g$, results in an illegal ordered $D_6$-partition, since the zero-block $B_0=\{\pm 1\}$ consists of only one pair. Toggling the sign of $\g(1)$, and moving the artificial separator before position $1$, we obtain:
$$\g'=[ \ \| \ 1,3,4 \mid -2 \mid -6,-5\mid \ ] \in B_6\setminus D_6,$$
that is transformed into the ordered $D_6$-partition:
$$[\{1,3,4\},\{-1,-3,-4\},\{-2\},\{2\},\{-6,-5\},\{6,5\}].$$
\end{exa}

\medskip

As in type $B$, it is easy to see that by applying the $D$-procedure to all the permutations in $D_n$, we obtain disjoint sets of ordered $D_n$-partitions, though, in this case we do not obtain all of them.
The next lemma specifies exactly which $D_n$- partitions are not reached:

\begin{lem}\label{structure of odd partitions}
The ordered $D_n$-partitions with $r$ pairs of nonzero blocks, which cannot be obtained by the $D$-procedure are exactly those of the form
\begin{equation}\label{odd_partitions}
\lambda = [C_1=\{*\}, -C_1=-\{*\}, C_2, -C_2,\ldots,C_r,-C_r],
\end{equation}
where $*$ stands for a single element of $[\pm n]$, and such that the total number
of negative entries in the blocks $C_1=\{*\},C_2, \dots, C_r$ is odd.
\end{lem}

\begin{proof}
First of all, we remark that when the $D$-procedure is applied to a permutation (equipped with separators) without the use of the switch operation, it produces ordered $D_n$-partitions $[C_0, C_1,-C_1,\dots, C_r,-C_r]$ with an even number of negative entries in the union $C_1 \cup C_2 \cup \cdots \cup C_r$. Let us call an ordered  $D_n$-partition {\em even} ({\it resp.} {\em odd}) if it satisfies ({\it resp}. does not satisfy) the latter condition.

In contrast, if the switch operation is applied, only odd partitions of the form $[C_1,-C_1,\dots, C_r,-C_r]$ without a zero-block are obtained, and the first block $C_1$ contains at least the two entries $\g(1)$ and $\g(2)$.

From this it follows that the partitions in \eqref{odd_partitions} cannot be reached.

Let us show, that all  other ordered $D_n$-partitions can be obtained using the  $D$-procedure.

Let $\lambda=[C_0,C_1,-C_1,\dots,C_r,-C_r]$ be an ordered $D_n$-partition with a non-empty zero-block $C_0$.
We look for the preimage $\g \in D_n$ of $\lambda$. Since the switch operation on a permutation $\g \in D_n$  produces $D_n$-partitions without a zero-block, in our case the switch operation has not been applied to $\g$.

We start by defining a sequence ${\mathbf S}$ as follows: we first put the positive entries of $C_0$ in their natural increasing order as the first elements of  $\mathbf{S}$, followed by a separator. If $\lambda$ is odd, we change the sign of the first entry of $\mathbf{S}$ to be negative.
Now, as described in the proof of Theorem~\ref{main_thm_B}, we complete $\mathbf{S}$ by concatenating the $r$ sequences composed by the elements of the blocks $C_1,\ldots,C_r$, where in each block the elements are ordered increasingly and followed by a separator.
We now consider the obtained sequence $\mathbf{S}$ as a permutation $\g \in D_n$. Note that $0 \notin {\rm Des}_D(\g)$, since by construction $|\g(1)| < \g(2)$ and so $\g(1) + \g(2) > 0$. Moreover, it is clear that applying the $D$-procedure to $\g$ yields the partition $\lambda$.

\smallskip

Now assume that $\lambda=[C_1,-C_1,\dots,C_r,-C_r]$ is an ordered $D_n$-partition without a zero-block.

If $\lambda$ is even, it is easy to see that the above construction without the initial step of reordering $C_0$, produces  $\g \in D_n$ which is the preimage of $\lambda$.

Finally, if $\lambda$ is odd and is not listed in Equation (\ref{odd_partitions})
it means that the first block $C_1$ has at least two elements,
and that the switch operation is necessary (due to the parity).
As before, we define a sequence $\mathbf{S}$ by reordering increasingly all the blocks $C_i$. Since $C_1$ has at least two elements, we have that  $\mathbf{S}(1)<\mathbf{S}(2)$.
Since the partition is odd with no zero-block, we have applied a switch operation on its preimage. Therefore, the sign of $\mathbf{S}(1)$ is negative. Now consider $\mathbf{S}$  as a permutation $\gamma \in D_n$. It is easy to see that the obtained permutation $\gamma \in D_n$ is indeed the preimage of $\lambda$.
\end{proof}

We give now two examples of the reverse procedure: both examples are ordered odd $D_5$-partitions, but one has a zero-block, while the other has no zero-block, so the latter requires the switch operation.

\begin{exa}\label{example for recovery}
(a) Let
$$\lambda_1=\left[ C_0=\{ \pm 1, \pm 4\}, \{3\}, \{-3\}, \{-5,2\}, \{5,-2\}  \right]$$
be an ordered odd $D_5$-partition with a zero-block $C_0$ which is odd since we have one negative sign in $\{3\}\cup \{-5,2\}$.
For recovering its preimage $\g_1 \in D_5$, we choose the negative sign for the smallest positive entry in the zero-block, which is $1$. After inserting the other positive entry of $C_0$ and a separator, we insert the other blocks, where each block is ordered increasingly followed by a separator, to obtain the permutation:
$$\g_1=[ \ -1,4\ |\ 3 \ | -5,2\ | \ ] \in D_5,$$
which is the preimage of the partition $\lambda_1$ using the $D$-procedure.

\medskip

\noindent
(b) Let $$\lambda_2= \left[ \{ -4, 3 \},\{4,-3\},\{2\},\{-2\},\{-5,-1\},\{5,1\} \right]$$
be an ordered odd $D_5$-partition without a zero-block. Hence, it is created by the switch operation. First, by the standard reverse procedure, we get the element:
$$\g_2'=\left[\ | -4,3 \ |\  2 \ | \ -5,-1 \ | \ \right] \in B_5 \setminus D_5.$$
Then, after performing the toggling of the sign of the first digit, we obtain:
$$\g_2=[ \ 4 \mid 3 \mid 2 \mid -5,-1 \ ]\in D_5,$$ that is the permutation from which the partition $\lambda_2$ is obtained.
Note that in this case, artificial separators are not needed.
\end{exa}

We can now complete the proof of Theorem~\ref{main_thm_D}.
\begin{proof}[Proof of Theorem ~\ref{main_thm_D}]
The equation in the statement of the theorem is equivalent to the following:
$$2^{r}r! S_D(n,r) = \sum\limits_{k=0}^r {A_D(n,k){\binom{n-k}{r-k}}}+  n \cdot 2^{n-1}(r-1)! \cdot S(n-1,r-1).$$

The left-hand side of the above equation counts the number of ordered $D_n$-partitions with $r$ pairs of nonzero blocks.
The right-hand side counts the same set of partitions divided in two categories: those coming from the $D$-procedure, that are  induced by permutations counted in $A_D(n,k)$, and those that are not, which are listed in Lemma \ref{structure of odd partitions}.
It is easy to see that the latter can be enumerated in the following way: first choose the absolute value of the unique element in $C_1=\{ * \}$, which can be done in $n$ ways. Then, choose and order the $r-1$ remaining blocks, which can be done in $(r-1)! \cdot S(n-1,r-1)$ ways. Finally, choose the sign of each entry in the blocks $C_1,C_2,\dots, C_r$, in such a way that an odd number of entries will be signed, and this can be done in $2^{n-1}$ ways.
This completes the proof.
\end{proof}

\section{Falling factorials for Coxeter groups \\of types $B$ and $D$}\label{Falling polynomials for Coxeter groups}

In this section, we present generalizations of Theorem \ref{thm:typeA_falling} for Coxeter groups of types $B$ and $D$ and provide combinatorial proofs for them.

\subsection{Type $B$}
The following theorem is a natural generalization of Theorem~\ref{thm:typeA_falling} for the Stirling numbers of type $B$, and it is a particular case of an identity appearing in Bala ~\cite{Bala}, where the numbers $S_B(n,k)$ correspond to the sequence denoted there by $S_{(2,0,1)}$. Bala uses generating functions techniques for proving this identity.

\begin{thm}[Bala]\label{thm_bala}
Let $x \in \mathbb{R}$ and let $n \in \mathbb{N}$. Then we have
\begin{equation}\label{B}
x^n=\sum\limits_{k=0}^n{S_B(n,k)[x]^B_k},
\end{equation}
where $[x]^B_k:=(x-1)(x-3)\cdots (x-2k+1)$ and $[x]^B_0:=1$.
\end{thm}

A combinatorial interpretation of $S_{B}(n,k)$ using the model of $k$-attacking rooks was given by Remmel and Wachs \cite{RW} (specifically, this is $S_{n,k}^{0,2}(1,1)$ in their notation).
More information on the rook interpretation of this and other factorization theorems can be found in Miceli and Remmel \cite{MR}.

Here we provide a kind of a geometric proof, suggested to us by Reiner, which is related to a method used by Blass and Sagan~\cite{BS} to compute the characteristic polynomial of the poset $L(\mathcal{B}_n)$.

\begin{proof}
Being a polynomial identity, it is sufficient to prove it only for odd integers $x=2m+1$ where $m \in \mathbb{N}$.

The left-hand side of Equation (\ref{B}) counts the number of lattice points in the $n$-dimensional cube $\{-m,-m+1,\dots,-1,0,1,\dots,m\}^n$.  We show that the right-hand side of Equation (\ref{B}) counts the same set of points using the maximal intersection subsets of hyperplanes the points lay on.

More precisely, let $\lambda=\{C_0,\pm C_1,\dots,\pm C_k\}$ be a $B_n$-partition with $k$ pairs of nonzero blocks, with
$0\leq k \leq n$. We associate to this partition the set of lattice points of the form $(x_1,\dots,x_n)$, where $x_j=0$ for all $j \in C_0$, and $x_{j_1}= x_{j_2}\neq 0$ ({\em resp.} $x_ {j_1}=-x_{j_2}\neq 0$) whenever $j_1,j_2$ ({\em resp.} $j_1,-j_2$) belong to the same block $C_i$ ({\em resp.} $-C_i$).

For the first pair of nonzero blocks $\pm C_1$ of the set partition $\lambda$, if $j_1 \in C_1 \cup -C_1$ then  there are $x-1$ possibilities (excluding the value $0$) to choose the value of $x_{j_1}$. For the second pair of blocks $\pm C_2$ of the partition $\lambda$, we have $x-3$ possibilities (excluding the value $0$ and the value $x_{j_1}$ chosen for $\pm C_1$ and its negative). We continue in this way until we get $x-(2k-3)$ possibilities for the last pair of blocks $\pm C_k$.

In particular, for $k=0$, $\lambda$ consists of only the zero-block $\{\pm 1,\dots,\pm n\}$, and is associated with the single lattice point $(0,\dots,0)$; for $k=n$, the only $B_n$-partition having $n$ pairs of nonzero blocks is
$$\{ \pm\{1\},\dots,\pm\{n\}\}$$
which corresponds to the lattice points $(x_1,\dots,x_n)$ such that
 $x_i \neq \pm x_j\neq 0$ for all $i\neq j$.
 Note that these are the $(x-1)(x-3)\cdots (x-(2n-1))$ lattice points that do not lie on any hyperplane.
\end{proof}

\begin{exa}
Let $n=2$ and $m=3$, so we have that $x=2m+1=7$. The lattice $([-3,3] \times [-3,3]) \cap \mathbb{Z}^2$ is presented in Figure \ref{fig_typeB}.

\begin{figure}[!ht]
    \centering
    \includegraphics[scale=0.5]{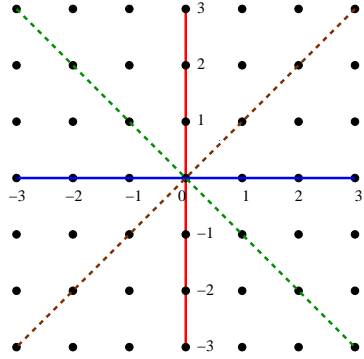}
    \caption{Lattice points for type $B$}
    \label{fig_typeB}
\end{figure}

For $k=0$, we have exactly one $B_2$-partition $\lambda_0$ consisting only of the zero-block: $\lambda_0=\{\{\pm 1,\pm 2\}\}$. The corresponding subspace is $\{x_1=x_2=0\}$, which counts only the lattice point $(0,0)$.

For $k=1$, we have four $B_2$-partitions, two of them contain a zero-block:
$$\lambda_1 = \{\{\pm 1\}, \pm\{2\} \}; \qquad \lambda_2 = \{\{\pm 2\},\pm\{1\} \},$$
and two of them do not:
$$\lambda_3 = \{ \pm\{1,2\} \}; \qquad \lambda_4 = \{ \pm\{1,-2\} \}.$$
The partitions $\lambda_1$ and $\lambda_2$ correspond to the axes $x_1=0$ and $x_2=0$, respectively.
The second pair $\lambda_3$ and $\lambda_4$ corresponds to the diagonals $x_1=x_2$ and $x_1=-x_2$ respectively. Each of these hyperplanes contains $6$ lattice points (since the origin is excluded).

For $k=2$, the single $B_2$-partition:
$$\lambda_5 = \{ \pm \{1\},\pm\{2\}\}$$
corresponds to the set of lattice points $(x_1,x_2)$ with $x_1\neq \pm x_2 \neq 0$, which are those not lying on any hyperplane.
\end{exa}

\begin{rem}
Note that Blass and Sagan \cite[Theorem 2.1]{BS} show that, when $x$ is an odd number, the cardinality of the set of lattice points not lying on any hyperplane is counted by the characteristic polynomial $\chi(\mathcal{B}_n,x)$ of the lattice $L(\mathcal{B}_n)$.
\end{rem}

\subsection{Type $D$}
The {\em falling factorial in type} $D$ is defined as follows: (see \cite{BS})
$$[x]_k^D:=\left\{ \begin{array}{ll}
1, & k=0 ;\\
(x-1)(x-3)\cdots (x-(2k-1)), & 1 \leq k <n ;\\
(x-1)(x-3)\cdots (x-(2n-3))(x-(n-1)),& k=n.\end{array}\right.$$

We have found no generalization of Equation (\ref{reg}) for type $D$ in the literature, so we supply one here.

\begin{thm}\label{thm_bala_D}
For all $n \in \mathbb{N}$ and $x \in \mathbb{R}$:
\begin{equation}\label{D}
x^n=\sum\limits_{k=0}^n{S_D(n,k)[x]_k^D} + n \left((x-1)^{n-1} -[x]_{n-1}^D\right).
\end{equation}
\end{thm}

\begin{proof}
For $D_n$-partitions having $0 \leq k< n$ pairs of nonzero blocks the proof goes verbatim as in type $B$, so let $k=n$.

In this case, we have only one possible $D_n$-partition having $n$ pairs of nonzero blocks:  $\{\pm\{1\}, \dots, \pm\{n\} \}$.
We associate this $D_n$-partition with the lattice points of the form $(x_1,\dots,x_n)$ such that $x_i \neq \pm x_j$ for $i \neq j$, having at most one appearance of the value $0$.
Note that the points with exactly one appearance of $0$ cannot be obtained by any $D_n$-partition having $k<n$ blocks, since the zero-block cannot consist of exactly one pair.
If $0$ does appear, then we have to place it in one of the $n$ coordinates and then we are left with $(x-1)(x-3)\cdots (x-(2n-3))$ possibilities for the rest, while if $0$ does not exist, then we have $(x-1)(x-3)\cdots (x-(2n-1))$ possibilities. These two values sum up to a total of $$[x]_n^D=(x-1)(x-3)\cdots (x-(2n-3))(x-(n-1)).$$
As in type $B$, this number is equal to the evaluation of the characteristic polynomial $\chi(\mathcal{D}_n,x)$ of $L(\mathcal{D}_n)$, where $x$ is odd.

Note that during the above process of collecting lattice points of the $n$-dimensional cube, the points containing exactly one appearance of $0$ and at least two nonzero coordinates are assigned the same absolute value are not counted, since the zero-block (if exists) must contain at least two elements.  This phenomenon happens when $n>2$, and the number of such points is $n((x-1)^{n-1} - [x]_{n-1}^D)$. This concludes the proof.\end{proof}

\begin{exa}
As in the previous example, let $n=2$ and $m=3$, so we have: $x=2m+1=7$. The lattice $([-3,3] \times [-3,3]) \cap \mathbb{Z}^2$ is presented in Figure \ref{fig_typeD}.

\medskip

\begin{figure}[!ht]
    \centering
    \includegraphics[scale=0.5]{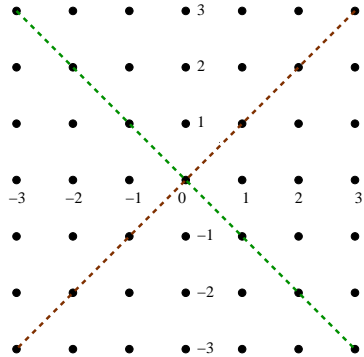}
    \caption{Lattice points for type $D$}
    \label{fig_typeD}
\end{figure}

\medskip

For $k=0$, as in type $B$ we have exactly one $D_2$-partition
$\lambda_0=\{\{\pm 1,\pm 2\}\}$ which counts only the lattice point $(0,0)$.
For $k=1$, we have only two $D_2$-partitions:
$\{\pm\{1,2\} \}$ and $\{\pm\{1,-2\} \}$, which correspond, as in the previous example,
to the diagonals  $x_1=x_2$ and $x_1=-x_2$ (without the origin), respectively

For $k=2$, as before, there is a single $D_n$-partition with two pairs of nonzero blocks:
$ \lambda = \{ \pm \{1\},\pm \{2\}\}.$
The lattice points corresponding to this set partition are those with different values in their coordinates, i.e. $x_1\neq x_2$, but in the case of type $D$ (in contrast to type $B$) the value $0$ can also appear. In the figure, these are all the lattice points which do not lie on the diagonals.

Note that in the case $n=2$ the second term in Equation \eqref{D} is $0$ and hence does not count any missing lattice points, since we have already counted all the points.  The missing points start to appear from $n=3$, as presented in the next example.
\end{exa}

\begin{exa}
Let  $n=m=3$, so that $x=2m+1=7$. The lattice points which are not counted have the form $(x_1,x_2,x_3)$, such that exactly one of their coordinates is $0$ and the other two share the same absolute value, e.g. the lattice points $(0,2,2)$ and $(0,2,-2)$ are not counted. In this case, the number of such missing lattice points (which is the first summand in the right-hand side of Equation \eqref{D}) is: $3\cdot 6^2-3\cdot 6 \cdot 4 =36$.
\end{exa}

\section{Some generalizations} \label{Possible generalizations}
In this section, we present some generalizations and variants related to our main results in some different directions. In Section \ref{grn_section}, we start with a short introduction to the colored permutations group and we generalize Theorems \ref{main_thm_B} and \ref{thm_bala} to this case. In Section \ref{flag_des_section}, we provide a version of Theorem \ref{main_thm_B} for the flag descent parameter in type $B$.

\subsection{The colored permutations group} \label{grn_section}

\bde
Let $m$ and $n$ be positive integers. {\it The group of
colored permutations of $n$ digits with $m$ colors} is the wreath product
$$G_{m,n}=\mathbb{Z}_m \wr S_n=\mathbb{Z}_m^n \rtimes S_n,$$
consisting of all the pairs $(\vec{z},\tau)$, where
$\vec{z}$ is an $n$-tuple of integers between $0$ and $m-1$ and $\tau \in S_n$.
\ede

A convenient way to look at $\grn$ is to consider the alphabet
$$\Sigma:=\left\{1,\dots,n, 1^{[1]},\dots,n^{[1]}, \dots, 1^{[m-1]},\dots,n^{[m-1]} \right\},$$
as the set $[n]$ colored by the colors
$0,\dots,m-1$. Then, an element of $\grn$ is a bijection $\pi: \Sigma \rightarrow \Sigma$
satisfying the following condition:
$$\mbox{if } \pi (i^{[\alpha]})=j^{[\beta]}\mbox{, then }
\pi ( i^{[\alpha+1]})=j^{[\beta+1]},$$
where the exponent $[\cdot]$ is computed modulo $m$. The elements of $G_{m,n}$ are usually called {\em colored permutations}.

In particular, $G_{1,n}=S_n$ is the symmetric group, while $G_{2,n}=B_n$ is the group of signed permutations.

\bde
The {\it color order} on $\Sigma$ is defined to be
$$1^{[m-1]} \prec  \cdots \prec n^{[m-1]} \prec \cdots \prec 1^{[1]} \prec 2^{[1]} \prec \cdots \prec n^{[1]}  \prec 1 \prec \cdots \prec n.$$
\ede

\begin{defn}
Let $\pi \in G_{m,n}$. Assume that $\pi(1)=a_1^{[z_1]}$. We define $${\rm des}_G(\pi):={\rm des}_A(\pi)+\varepsilon(\pi),$$
where
\begin{equation}\label{desA}
{\rm des}_A(\pi):=|\{i \in [n-1] \mid \pi(i) \succ \pi(i+1)\}|,
\end{equation}
where `$\succ$' refers to the color order, and
\begin{equation}\label{epsilon}
\varepsilon(\pi):=\left\{\begin{array}{cc}
     1, &  {\rm if} \ z_1 \not \equiv 0 \ {\rm mod} \ m;\\
     0, &  {\rm if} \ z_1 \equiv 0  \ {\rm mod} \ m.
\end{array}\right.
\end{equation}
For example, if $\pi=[3,\bar{1},\bar{\bar{2}}] \in G_{3,3}$, we have ${\rm des}_G(\pi)= 2+0=2.$ Note that for $m=1$, $\des_G=\des$ and for $m=2$, $\des_G=\des_B$.

Moreover, we define the {\it Eulerian number of type $G_{m,n}$} to be:
$$A_m(n,k):=|\{\pi \in G_{m,n} \mid {\rm des}_G(\pi)=k\}|.$$
\end{defn}

Let $C\subseteq \Sigma$. Denote $C^{[t]}=\{x^{[i+t]}\mid x^{[i]} \in C\}$.
\begin{defn}
A {\it $\grn$-partition} is a set partition of $\Sigma$
into blocks such that the following conditions are satisfied:
\begin{itemize}
\item There exists at most one block satisfying $C^{[1]}=C$.
This block will be called the {\it zero-block}.
\item If $C$ appears as a block in the partition, then $C^{[1]}$ also appears in that partition.
\end{itemize}

Two blocks $C_1$ and $C_2$ will be called {\it equivalent} if there is a natural number $t \in \mathbb{N}$ such that $C_1=C_2^{[t]}$.

The number of $\grn$-partitions with $r$ non-equivalent nonzero blocks is denoted by $S_{m}(n,r)$.
\end{defn}
For example, the following is a $G_{3,4}-$ partition:
$$\{\{1,\bar{1},\bar{\bar{1}},2,\bar{2},\bar{\bar{2}}\},\{3,\bar{4}\},\{\bar{3},\bar{\bar{4}}\},\{\bar{\bar{3}},4\}\},$$ with a zero-block: $\{1,\bar{1},\bar{\bar{1}},2,\bar{2},\bar{\bar{2}}\}$.

We define now the concept of an {\it ordered} $\grn$-partition:

\begin{defn}
A $\grn$-partition is {\em ordered} if the set of blocks is totally ordered and the following conditions are satisfied:

\begin{itemize}
\item If the zero-block exists, then it appears as the first block.

\item For each nonzero block $C$, the blocks $C^{[i]}$ for $1 \leq i \leq m-1$ appear consecutively right after $C$, i.e. $C,C^{[1]}, C^{[2]},\dots, C^{[m-1]}$.
\end{itemize}
\end{defn}

The generalization of Theorem \ref{thm:typeA} in this setting is as follows.
\begin{thm}\label{main_thm_2}
For all positive integers $n,m$ and $r$, we have:
$$ S_m(n,r)= \frac{1}{m^r r!} \sum\limits_{k=0}^{r} {A_m(n,k){\binom{n-k}{r-k}}}.$$
\end{thm}

The proof is similar to that of Theorem \ref{main_thm_B}, so it is omitted.

\medskip

In order to generalize Theorem~\ref{thm:typeA_falling}, we define the {\em falling factorial of type} $\grn$ as follows: (see Equation 15 in \cite{Bala})
$$[x]_k^m:=\left\{ \begin{array}{ll}
1, & k=0 ;\\
(x-1)(x-1-m)\cdots (x-1-(k-1)m), & 1 \leq k \leq n.
\end{array}\right.$$

We have:
\begin{thm}\label{thm Balla for grn} Let $x \in \R$ and $n\in \mathbb{N}$. Then we have:
$$x^n=\sum\limits_{k=0}^n{S_m(n,k)[x]_k^m}.$$
\end{thm}

We present here the idea of the proof.
\begin{proof}[Sketch of the proof.]
Divide the unit circle $S^1$ in the plane into $m$ parts according to the $m$th roots of unity: $1,\rho_m,\rho_m^2,\dots,\rho_m^{m-1}$, see Figure \ref{S1}, where $m=3$ and the roots are represented by small bullets. This divides the circle into $m$ arcs. Now, in each arc, locate $t$ points in equal distances from each other (see Figure \ref{S1} where $t=5$ and the points are represented by small lines). Including the point $(1,0)$, we get $x=mt+1$ points on the unit circle.

Consider now the $n$-dimensional torus $(S^1)^n=S^1 \times \cdots \times S^1 $ with $x^n$ {\em lattice points on it}. The same arguments we presented in the proof of Theorem \ref{thm_bala} will apply now to Theorem \ref{thm Balla for grn}, when we interpret the $\grn$-partitions as intersections of subsets of hyperplanes in $\mathcal{G}_{m,n}$, where by $\mathcal{G}_{m,n}$ we mean the following generalized hyperplane arrangement for the colored permutations group:
\begin{eqnarray*}
\mathcal{G}_{m,n}&:=&\{ \ \{x_i=\rho_m^k x_j\} \mid 1 \leq i < j \leq n ,0 \leq k <m\} \\
&\cup& \{ \ \{x_i=0\} \mid 1 \leq i \leq n\},
\end{eqnarray*}
See e.g. \cite[p. 244]{OT}.
\end{proof}

\begin{figure}[!ht]
    \centering
    \begin{tikzpicture}[cap=round,line width=2pt]
  \draw  (0,0) circle (2cm);

   \foreach \angle in
    {20, 40, 60, 80, 100, 140, 160 , 180,
     200, 220, 260,280, 300, 320, 340}
  {
    \draw[line width=2pt] (\angle:1.9cm) -- (\angle:2.1cm);
  }

  \foreach \angle in {0,120,240}
    \draw[line width=2pt,fill] (\angle:2cm) circle [radius=0.1];

\draw[->, line width=1pt] (-2,0) -- (3,0);
\draw[->, line width=1pt] (0,-3) -- (0,3);
\node[above right] at (2cm,0){\tiny $\rho_3^0=1$};
\node[right] at (20:2.1cm){\tiny $1$};
\node[right] at (40:2.1cm){\tiny $2$};
\node[above right] at (60:2cm){\tiny $3$};
\node[above] at (80:2.1cm){\tiny $4$};
\node[above] at (100:2.1cm){\tiny $5$};

\node[above left] at (120:2cm){\tiny $\rho_3^1$};
\node[left] at (140:2.1cm){\tiny $\rho_3^1 \cdot 1$};
\node[left] at (160:2.1cm){\tiny $\rho_3^1 \cdot 2$};
\node[left] at (180:2.1cm){\tiny $\rho_3^1 \cdot 3$};
\node[left] at (200:2.1cm){\tiny $\rho_3^1 \cdot 4$};
\node[below left] at (220:2cm){\tiny $\rho_3^1 \cdot 5$};

\node[below left] at (240:2cm){\tiny $\rho_3^2$};
\node[below] at (260:2.1cm){\tiny \hspace{-5pt}$\rho_3^2 \cdot 1$};
\node[below] at (280:2.1cm){\tiny \hspace{5pt}$\rho_3^2 \cdot 2$};
\node[below right] at (300:2cm){\tiny $\rho_3^2 \cdot 3$};
\node[right] at (320:2.1cm){\tiny $\rho_3^2 \cdot 4$};
\node[right] at (340:2.1cm){\tiny $\rho_3^2 \cdot 5$};

\end{tikzpicture}
\caption{The 16 lattice points on $S^1$, representing the first coordinate for $m=3$ and $t=5$.}
    \label{S1}
\end{figure}
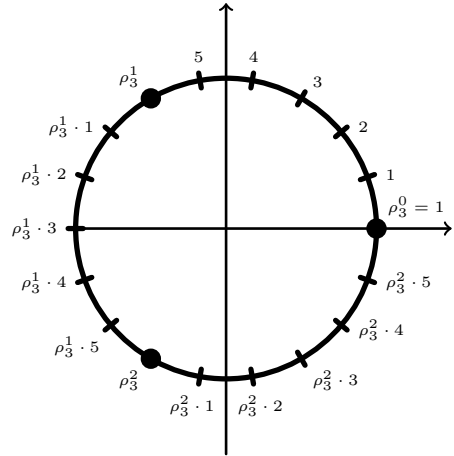

\subsection{The flag descent parameter for the Coxeter group of type~$B$}\label{flag_des_section}

Another possibility to generalize these results is to consider the {\em flag descent statistics} defined on group of signed permutations.
Such parameters produce, in this context, similar expressions of those presented in the previous sections, but less elegant.
As an example, we show here only one of these possible extensions.

This is a different generalization of Theorem~ \ref{main_thm_B} by using the {\em flag-descent number} fdes, that is defined in~\cite{AR} for a signed permutation $\be \in B_n$:
$${\rm fdes}(\be):=2 \cdot {\des_A}(\be)+\varepsilon(\be).$$
where $\des_A(\be)$ is defined as in Equation \eqref{desA}, and $\varepsilon(\be)$ as in Equation \eqref{epsilon}.
\smallskip

We denote by $A^*_B(n,k)$ the number of permutations $\be \in B_n$ satisfying ${\rm fdes}(\be)=k-1$, and by
$S^*_{B}(n,r)$ the number of $B_n$-partitions having exactly $r$ blocks. Here, differently from $S_B(n,r)$, every block counts: the zero-block is counted once, and any pair $\pm C_i$ is counted twice.

\medskip

These two new parameters satisfy the identity stated below:

\begin{thm}\label{thm_flag}
For all nonnegative integers $n,r$ where $n \geq r$, we have:
$$2^{\left\lfloor \frac{r}{2}\right\rfloor} \left\lfloor \frac{r}{2} \right\rfloor ! \ S^*_{B}(n,r)= \sum\limits_{k=1}^r {A^*_B(n,k)\binom{n- \left\lceil\frac{k}{2}\right\rceil}{  \left\lfloor\frac{r-k}{2}\right\rfloor}}.$$
\end{thm}

The proof uses arguments similar to those in the proof of Theorem~\ref{main_thm_B}, and is therefore omitted.

\end{document}